\documentclass[11pt,a4paper]{amsart}
\usepackage[utf8]{inputenc}
\usepackage[T1]{fontenc}
\usepackage{times}

\usepackage{amssymb}
\usepackage{amsmath}
\usepackage{amsfonts}
\usepackage{amssymb}
\usepackage{amsthm}
\usepackage{enumerate}

\newtheorem{theorem}{Theorem}[section]
\newtheorem{lemma}[theorem]{Lemma}
\newtheorem{proposition}[theorem]{Proposition}

\theoremstyle{definition}

\newtheorem{remark}[theorem]{Remark}

\newcommand{\wh}{\widehat}

\newcommand{\R}{\mathbb{R}}
\newcommand{\N}{\mathbb{N}}

\newcommand{\cC}{{\mathcal C}}

\newcommand{\cN}{{\mathcal N}}

\newcommand{\cR}{{\mathcal R}}   
\newcommand{\cS}{{\mathcal S}}

\newcommand{\mR}{{\mathbf R}}

\newcommand{\eps}{\varepsilon}

\newcommand{\weakto}{\rightharpoonup}

\newcommand{\embed}{\hookrightarrow}

\begin{document}

\title[The nonlinear Helmholtz equation in the high frequency limit]{Multiple standing waves 
	for the nonlinear Helmholtz equation concentrating in the high frequency limit}
\author{Gilles Ev\'equoz}
\address{Goethe-Universit\"at Frankfurt, Institut f\"ur Mathematik, 
		Robert-Mayer-Str.~10,  D-60325 Frankfurt am Main, Germany} 
\email{evequoz@math.uni-frankfurt.de}

\begin{abstract}
This paper studies for large frequency number $k>0$ the existence and multiplicity
of solutions of the semilinear problem
$$
-\Delta u -k^2 u=Q(x)|u|^{p-2}u\quad\text{ in }\R^N, \quad N\geq 2.
$$
The exponent $p$ is subcritical and the coefficient $Q$ is continuous, nonnegative and
satisfies the condition
$$
\limsup_{|x|\to\infty}Q(x)<\sup_{x\in\R^N}Q(x).
$$
In the limit $k\to\infty$, sequences of solutions associated to ground states of a dual equation 
are shown to concentrate, after rescaling, at global maximum points of the function $Q$. 
\end{abstract}

\maketitle

\section{Introduction and main results}
The existence of solutions of semilinear elliptic PDEs on $\R^N$, 
concentrating at single points or on higher-dimensional sets has a long history. 
In their pioneering papers, Floer and Weinstein~\cite{floer.weinstein:86} and
Rabinowitz \cite{rabinowitz91} studied this question
for positive solutions of the nonlinear Schr\"odinger equation 
\begin{equation}\label{eqn:NLS_eps}
-\eps^2\Delta u + V(x) u = Q(x)|u|^{p-2}u \quad\text{ in }\ \R^N,
\end{equation}
in the case where $Q\equiv 1$ and assuming $\inf V>0$. Under the global condition
\begin{equation}\label{eqn:liminf_V}
\liminf_{|x|\to\infty}V(x)>\inf_{x\in\R^N}V(x),
\end{equation}
it was proved in \cite{rabinowitz91} that a ground state 
(i.e., positive least-energy solution) of \eqref{eqn:NLS_eps} exists for small $\eps>0$.
In the limit $\eps\to 0$, Wang \cite{wang93} showed that sequences of ground states concentrate 
at a global minimum point $x_0$ of $V$ and converge, after rescaling, 
towards the ground state of the limit problem
\begin{equation}\label{eqn:NLS_eps_limit}
-\Delta u + V(x_0)u=|u|^{p-2}u \quad\text{ in }\ \R^N.
\end{equation}
Extensions of these results were obtained by many authors and the interested reader
may consult the monograph by Ambrosetti and Malchiodi \cite{ambrosetti-malchiodi-a} 
for a precise list of references. Among the recent papers on this topic, 
let us point out the work of Byeon, Jeanjean and Tanaka \cite{byeon-jeanjean,byeon-jeanjean-tanaka} 
where the right-hand side is replaced by a very broad class of autonomous nonlinearities
and the paper by Bonheure and Van Schaftingen \cite{bonheure-vanschaftingen08} 
in which $V$ is allowed to vanish at infinity and $Q$ may have singularities.

In the present paper, we focus on the nonlinear Helmholtz equation
\begin{equation}\label{eqn:nlh}
-\Delta u -k^2 u =Q(x)|u|^{p-2}u \quad\text{in }\ \R^N,
\end{equation}
where $Q\geq 0$ is a bounded function. Our aim is to investigate the existence of real-valued
solutions for $k>0$ large, as well as their behavior as $k\to\infty$.
Setting $\eps=k^{-1}$ and $w=\eps^{\frac2{p-2}}u$, we find that $w$ solves the problem \eqref{eqn:NLS_eps}
with $V\equiv -1$ and it is therefore natural to ask, whether the concentration results mentioned above
can also be obtained for this equation. But when trying to adapt the previous methods to the present case,
several obstacles appear. First, the structure of the limit problem
\begin{equation}\label{eqn:NLH_eps_limit}
-\Delta u -u =Q(x_0)|u|^{p-2}u\quad\text{in }\ \R^N
\end{equation}
is more complex than \eqref{eqn:NLS_eps_limit}. In particular all 
solutions of \eqref{eqn:NLH_eps_limit} change sign infinitely many times
and no uniqueness result is known. Second, there is no direct variational formulation available 
for the problems \eqref{eqn:nlh} -- \eqref{eqn:NLH_eps_limit} 
and therefore no natural concept of ground state associated to them.
Nevertheless, we will show that variational arguments in the spirit of
\cite{rabinowitz91,wang93} can be used to obtain existence and concentration results 
for solutions of the nonlinear Helmholtz equation \eqref{eqn:nlh}.

Our method relies on the dual variational framework established in the recent paper \cite{evequoz-weth-dual}
which consists in inverting the linear part and the nonlinearity. More precisely,
setting $\eps=\frac1k$ and $Q_\eps(x)=Q(\eps x)$, we look at the integral equation
\begin{equation}\label{eqn:dual} 
|v|^{p'-2}v=Q_\eps^{\frac1p}\mR\left(Q_\eps^{\frac1p}v\right),
\end{equation}  
where $p'=\frac{p}{p-1}$ and where $\mR$ denotes the real part of the 
Helmholtz resolvent operator.
The solutions of this equation are critical points of the so-called dual energy functional 
$J_\eps$: $L^{p'}(\R^N)$ $\to$ $\R$ given by
$$
J_\eps(v)=\frac1{p'}\int_{\R^N}|v|^{p'}\, dx -\frac12\int_{\R^N} Q_\eps^{\frac1p}v\mR\left(Q^{\frac1p}v\right)\, dx
$$
and every critical point $v$ of $J_\eps$ gives rise to a strong solution $u$ of \eqref{eqn:nlh} with $k=\frac1\eps$
by setting
\begin{equation}\label{eqn:u_transf_dual_introd}
u(x)=k^{\frac2{p-2}}\mR\left(Q_\eps^\frac1pv\right)(kx), \quad x\in\R^N.
\end{equation}
This correspondence allows us to define a notion of ground state for \eqref{eqn:nlh} as follows.
If $\eps=\frac1k$ and $v$ is a nontrivial critical point for $J_\eps$ at the mountain pass level,
the function $u$ given by \eqref{eqn:u_transf_dual_introd} will be called 
a \emph{dual ground state} of \eqref{eqn:nlh}.

A motivation behind this definition is given by considering \eqref{eqn:nlh} on a bounded domain 
with Dirichlet boundary condition. For this problem, Szulkin and Weth \cite[Sect.~3]{szulkin-weth09} 
proved that the ground state level for the direct functional is attained by a nontrivial critical point. 
In the case where the linear operator $-\Delta-k^2$ is invertible, one can show that 
it is also a critical point of the dual energy functional at the mountain pass level.

The first main result of this paper concerns the existence and 
concentration, up to rescaling, of sequences of dual ground states.
\begin{theorem}\label{thm1}
Let $N\geq 2$, $\frac{2(N+1)}{N-1}<p<\frac{2N}{N-2}$ (resp. $6<p<\infty$ if $N=2$) 
and consider a bounded continuous function $Q\geq 0$ such that
\begin{equation}\label{eqn:global_cond}
Q_\infty:=\limsup\limits_{|x|\to\infty}Q(x)<Q_0:=\sup\limits_{x\in\R^N}Q(x).
\end{equation}
\begin{itemize}
\item[(i)] There is $k_0>0$ such that for all $k>k_0$ the problem \eqref{eqn:nlh} admits a dual
ground state.

\item[(ii)] Let $(k_n)_n\subset(k_0,\infty)$ satisfy $\lim\limits_{n\to\infty}k_n=\infty$ 
and consider for each $n$, a dual ground state $u_n$ of
$$
-\Delta u + k_n u=Q(x)|u|^{p-2}u \quad\text{in }\ \R^N.
$$
Then there is a maximum point $x_0$ of $Q$, a dual ground state $u_0$ of
\begin{equation}\label{eqn:limit_Q0}
-\Delta u -u =Q_0|u|^{p-2}u \quad\text{in }\ \R^N
\end{equation}
and a sequence $(x_n)_n\subset\R^N$ such that (up to a subsequence)
$\lim\limits_{n\to\infty}x_n= x_0$ and  
$$
k_n^{-\frac{2}{p-2}}u_n\left(\frac{\cdot}{k_n}+x_n\right)\to u_0\quad\text{ in }\ L^p(\R^N),\ 
\text{ as }n\to\infty.
$$ 
\end{itemize}
\end{theorem}
In the Schr\"odinger equation \eqref{eqn:NLS_eps} with $V\equiv 1$, Wang and Zeng \cite{wang-zeng97} noticed that 
\eqref{eqn:global_cond} plays the same role as the Rabinowitz condition \eqref{eqn:liminf_V}. 
As a consequence of Theorem \ref{thm1}, we see that this condition also ensures the concentration, 
in the $L^p$-sense, for \eqref{eqn:NLS_eps} with $V\equiv -1$. To the best of our knowledge, this is the first 
concentration result for semilinear problems where $0$ lies in the interior of the essential spectrum 
of the linearization.

The proof of the above theorem is given in Section 3. It relies on the fact that, due to \eqref{eqn:global_cond}, 
the dual energy functional satisfies the Palais-Smale condition at all levels strictly
below the least among all possible energy levels for the problem at infinity. 
In contrast to similar problems where the dual method is used (see, e.g., \cite{alves-yang-soares03}), 
we have no sign information about the nonlocal term appearing in the dual energy functional,
since the resolvent Helmholtz operator is not positive definite. In order to handle this term, 
we derive a new energy estimate (Lemma \ref{lem:interaction}) for the nonlocal interaction between 
functions with disjoint support, which we believe to be of independent interest. 
The proof of the $L^p$-concentration in Part (ii) of the above theorem is given in Theorem \ref{thm:conc_gs}. 
The main ingredients are an energy comparison with the limit problem \eqref{eqn:limit_Q0} and a representation
lemma for Palais-Smale sequences (Lemma \ref{lem:splitting}) in the spirit of and Benci and Cerami \cite{benci-cerami87}.

The second main result in this paper is the following multiplicity result for
\eqref{eqn:nlh} with $k>0$ large. Here, $M=\{x\in\R^N\, :\, Q(x)=Q_0\}$
denotes the set of maximum points of $Q$, and for $\delta>0$ we let
$M_\delta=\{x\in\R^N\, :\, \text{dist}(x,M)\leq\delta\}$.
Also, for a closed subset $Y$ of a metric space $X$ we denote by $\text{cat}_X(Y)$ 
the Ljusternik-Schnirelman category of $Y$ with respect to $X$, i.e., the least number
of closed contractible sets in $X$ which cover $Y$.
\begin{theorem}\label{thm2}
Let $N\geq 2$, $\frac{2(N+1)}{N-1}<p<\frac{2N}{N-2}$ (resp. $6<p<\infty$ if $N=2$) 
and consider a bounded and continuous function $Q\geq 0$
satisfying \eqref{eqn:global_cond}. 
For every $\delta>0$, there exists $k(\delta)>0$ such that \eqref{eqn:nlh} 
has at least $\operatorname*{cat}_{M_\delta}(M)$
nontrivial solutions for all $k>k(\delta)$.
\end{theorem}
In the case where $Q_\infty=0$, Palais-Smale sequences for the dual functional are relatively compact
and a mountain pass argument was used in \cite{evequoz-weth-dual} to obtain the existence
of infinitely many solutions. When $Q_\infty>0$, only Palais-Smale sequences below the least-energy 
level at infinity are relatively compact. This loss of compactness has to be handled in order to
prove the existence of multiple solutions. Our proof uses topological arguments close to the ones developed by 
Cingolani and Lazzo \cite{cingolani-lazzo97} for \eqref{eqn:NLS_eps} (see also \cite{cingolani-lazzo00}) 
and based on ideas of Benci, Cerami and Passaseo
\cite{benci-cerami91, benci-cerami-passaseo91} for problems on bounded domains.
The main point lies in the construction of two maps whose composition is homotopic to the inclusion
$M\embed M_\delta$.
For more results concerning the multiplicity of solutions for small $\eps>0$
in the Schr\"odinger equation \eqref{eqn:NLS_eps} with $\inf V>0$,
the interested reader my consult the very recent paper by Cingolani, Jeanjean and Tanaka \cite{cingolani-jeanjean-tanaka2015}
and the references therein.

\bigskip

The paper is organized as follows. In Section 2, we describe the dual variational framework
set up in \cite{evequoz-weth-dual} for the study of the problem \eqref{eqn:nlh} with fixed $k$, and
discuss the basic properties of the associated Nehari manifold.
Next, we establish a representation lemma for Palais-Smale sequences of the dual
energy functional in the case of constant $Q$. 
The section concludes with the proof of the Palais-Smale condition for the dual energy functional
on the Nehari manifold below some limit energy level. A crucial element in the proof 
of this result is the decay estimate for the nonlocal interaction induced by the Helmholtz 
resolvent operator given in Lemma~\ref{lem:interaction}. 
In Section 3, we start by proving that for small $\eps=k^{-1}>0$ the least-energy level for critical points 
of the dual energy functional is attained (Proposition \ref{prop:c_eps_attained}). As a consequence of this,
we obtain Part (i) in Theorem \ref{thm1}. In a second part, the concentration in the limit $\eps=k^{-1}\to 0$
is established for sequences of ground-states in the dual formulation (Proposition~\ref{prop:limit_gs})
and this allows us to prove Part (ii) in Theorem \ref{thm1}. The last section, Section 4, 
is devoted to the proof of Theorem \ref{thm2}.


\section{The variational framework}
\subsection{Notation and preliminaries}
Throughout the paper, we let $N\geq 2$ and consider a nonnegative function
$Q\in L^\infty(\R^N)\backslash\{0\}$. Setting $2_\ast:=\frac{2(N+1)}{N-1}$ and $2^\ast:=\frac{2N}{N-2}$ 
if $N\geq 3$, resp. $2^\ast:=\infty$ if $N=2$, we fix an exponent $p\in(2_\ast,2^\ast)$ and 
we let $p'=\frac{p}{p-1}$ denote its conjugate exponent.
For $1\leq q\leq\infty$, we write $\|\cdot\|_q$ instead of $\|\cdot\|_{L^q(\R^N)}$ 
for the standard norm of the Lebesgue space $L^q(\R^N)$. In addition, for $r>0$ and $x\in\R^N$, we denote
by $B_r(x)$ the open ball in $\R^N$ of radius $r$ centered at $x$, and let $B_r=B_r(0)$.

With this notation, we consider for $k>0$ the equation
\begin{equation}\label{eqn:sp}
-\Delta u -k^2u=Q(x)|u|^{p-2}u \quad\text{in }\ \R^N.
\end{equation}
Setting $\eps=k^{-1}$, $u_\eps(x)=\eps^{\frac2{p-2}}u(\eps x)$ and $Q_\eps(x)=Q(\eps x)$, $x\in\R^N$, 
\eqref{eqn:sp} can be rewritten as
\begin{equation}\label{eqn:sp2}
-\Delta u_\eps -u_\eps = Q_\eps(x)|u_\eps|^{p-2}u_\eps \quad\text{in }\ \R^N.
\end{equation}
Consider the fundamental solution
\begin{equation}\label{eqn:phi}
\Phi(x)=\frac{i}{4}(2\pi|x|)^{\frac{2-N}{2}}H^{(1)}_{\frac{N-2}{2}}(|x|), \quad x\in\R^N\backslash\{0\},
\end{equation}
of the Helmholtz equation $-\Delta u-u=0$, where $H^{(1)}_\nu$ denotes the Hankel function 
of the first kind of order $\nu$.
As a consequence of estimates by Kenig, Ruiz and Sogge~\cite{KRS87}, the operator $\mR$,
defined on the Schwartz space $\cS(\R^N)$ of rapidly decreasing functions 
by the convolution
$$
\mR f=\text{Re}(\Phi)\ast f, \quad f\in\cS(\R^N),
$$ 
has a continuous extension $\mR$: $L^{p'}(\R^N)$ $\to$ $L^p(\R^N)$.
Using this operator, we define
the $C^1$-functional $J_\eps$: $L^{p'}(\R^N)$ $\to$ $\R$,
$$
J_\eps(v):=\frac1{p'}\int_{\R^N}|v|^{p'}\, dx 
-\frac12\int_{\R^N}Q_\eps^\frac1pv\mR(Q_\eps^\frac1pv)\, dx, \quad v\in L^{p'}(\R^N)
$$
(for more details on the construction of $\mR$ and $J_\eps$, see \cite{evequoz-weth-dual}).
Every critical point of $J_\eps$ corresponds to a solution of \eqref{eqn:sp2} in the following way.
A function $v\in L^{p'}(\R^N)$ satisfies $J_\eps'(v)=0$ if and only if it solves the integral equation
$$
|v|^{p'-2}v=Q_\eps^\frac1p \mR(Q_\eps^\frac1p v).
$$
Setting $u=\mR(Q_\eps^\frac1p v)$, it is equivalent to
\begin{equation}\label{eqn:u_fp}
u= \mR(Q_\eps|u|^{p-2}u)
\end{equation}
and since $\mR$ is a right inverse for the Helmholtz operator $-\Delta-1$, it follows
that $u$ is a strong solution of \eqref{eqn:sp2} (see \cite[Lemma 4.3 and Theorem 4.4]{evequoz-weth-dual} 
concerning the regularity and asymptotic behavior of $u$).
Conversely, if $u$ solves \eqref{eqn:u_fp}, then $v=Q_\eps^\frac1{p'}|u|^{p-2}u$ is a critical point of $J_\eps$.
Notice that distinct critical points correspond to distinct solutions of \eqref{eqn:u_fp} and therefore of
\eqref{eqn:sp2}.

Let us now recall some properties of the dual functional, obtained 
in \cite{e-helm-2d,e-helmholtz-periodic-asympt,evequoz-weth-dual}.
Since $p'<2$ and since the kernel of the operator $\mR$ is positive close to the origin, 
the geometry of the functional $J_\eps$ is of mountain pass type:
\begin{align}
\exists\; \alpha>0\text{ and }\rho>0\text{ such that }J_\eps(v)\geq\alpha>0, 
\quad \forall v\in L^{p'}(\R^N)\text{ with }\|v\|_{p'}=\rho.\label{eqn:J_local_min}\\
\exists\; v_0\in L^{p'}(\R^N)\text{ such that }\|v_0\|_{p'}>\rho\text{ and }J_\eps(v_0)<0.\label{eqn:J_v0_neg}
\end{align}
As a consequence, the Nehari set associated to $J_\eps$:
$$
\cN_\eps:=\{v\in L^{p'}(\R^N)\backslash\{0\}\ :\ J'_\eps(v)v=0\},
$$
is not empty. More precisely, by \eqref{eqn:J_v0_neg}, the set
$$
U^+_\eps:=\left\{ v\in L^{p'}(\R^N)\, :\, \int_{\R^N}Q_\eps^\frac1pv\mR(Q_\eps^\frac1pv)\, dx>0\right\}
$$
is not empty and for each $v\in U^+_\eps$ there is a unique $t_v>0$ for which $t_vv\in\cN_\eps$ holds. 
It is given by
\begin{equation}\label{eqn:t_v}
t_v^{2-p'}=\frac{\int_{\R^N}|v|^{p'}\, dx}{\int_{\R^N}Q^\frac1pv\mR(Q^\frac1pv)\, dx}.
\end{equation}
In addition, $t_v$ is the unique maximum point of $t\mapsto J_\eps(tv)$, $t\geq 0$.
Using \eqref{eqn:J_local_min} we obtain in particular
$$
c_\eps:=\inf\limits_{\cN_\eps}J_\eps=\inf\limits_{v\in U_\eps^+}J_\eps(t_vv)>0.
$$
Moreover, for every $v\in \cN_\eps$ we have
$c_\eps\leq J_\eps(v)=\left(\frac1{p'}-\frac12\right)\|v\|_{p'}^{p'}$. Hence,
$0$ is isolated in the set $\{v\in L^{p'}(\R^N)\, :\, J_\eps'(v)v=0\}$ and, 
as a consequence, the $C^1$-submanifold $\cN_\eps$ of $L^{p'}(\R^N)$ is complete.

\bigskip

We recall that $(v_n)_n\subset L^{p'}(\R^N)$ is termed a Palais-Smale sequence,
or a (PS)-sequence, for $J_\eps$ if $(J_\eps(v_n))_n$ is bounded and $J'_\eps(v_n)\to 0$
as $n\to\infty$. Also, for $d>0$, we say that $(v_n)_n$ is a (PS)$_d$-sequence for $J_\eps$
if it is a (PS)-sequence and if $J_\eps(v_n)\to d$ as $n\to\infty$.
The following properties hold (see \cite[Sect.~2]{e-helmholtz-periodic-asympt}).
\begin{lemma}\label{lem:PS_sequences}
Let $(v_n)_n\subset L^{p'}(\R^N)$ be a Palais-Smale sequence for $J_\eps$. 
Then $(v_n)_n$ is bounded and there exists $v\in L^{p'}(\R^N)$ such that $J_\eps'(v)=0$ 
and (up to a subsequence) $v_n\weakto v$ weakly in $L^{p'}(\R^N)$ 
and $J_\eps(v)\leq\liminf\limits_{n\to\infty} J_\eps(v_n)$.

Moreover, for every bounded and measurable set $B\subset\R^N$, 
$1_Bv_n\to 1_Bv$ strongly in $L^{p'}(\R^N)$. 
\end{lemma}
As a consequence, we obtain the following characterization of 
the infimum $c_\eps$ of $J_\eps$ over the Nehari manifold $\cN_\eps$
(see \cite[Sect.~4]{e-helmholtz-periodic-asympt}).
\begin{lemma}\label{lem:nehari_mp}
(i) $c_\eps$ coincides with the mountain pass level, i.e.,
$$
c_\eps=\inf\limits_{\gamma\in\Gamma}\max\limits_{t\in[0,1]}J_\eps(\gamma(t)), \quad
\text{where}\quad \Gamma=\left\{\gamma\in C([0,1],L^{p'}(\R^N))\, :\, \gamma(0)=0
\text{ and }J_\eps(\gamma(1))<0\right\}.
$$
(ii) If $c_\eps$ is attained, then 
$c_\eps=\min\{J_\eps(v)\, :\, v\in L^{p'}(\R^N)\backslash\{0\}, \ J_\eps'(v)=0\}$.
\vspace{0.2cm}\\
(iii) If $Q_\eps(x)\to \inf\limits_{\R^N} Q$ as $|x|\to\infty$, then $c_\eps$ is attained.
\end{lemma}
In view of the preceding results, we introduce the following terminology.
If $v\in L^{p'}(\R^N)\backslash\{0\}$ is a critical point for $J_\eps$ at the
mountain pass level, i.e.  
$J_\eps'(v)=0$ and $J_\eps(v)=c_\eps$,
we call the function $u$ given by
\begin{equation}\label{eqn:u_transf_dual}
u(x)=k^{\frac2{p-2}}\mR\left(Q_\eps^\frac1pv\right)(kx), \quad x\in\R^N,
\end{equation}
where $k=\eps^{-1}$,
a \emph{dual ground state} of \eqref{eqn:sp}. More generally,
if $v$ is a nontrivial critical point of $J_\eps$, the function $u$ obtained from $v$
by \eqref{eqn:u_transf_dual} will be called a \emph{dual bound state} of \eqref{eqn:sp}.

\subsection{Representation lemma and Palais-Smale condition}
We now take a closer look at the Palais-Smale sequences of the functional $J_\eps$ and first
prove a representation lemma in the case where the coefficient $Q$ is a positive constant. 
A crucial ingredient related to the nonlocal quadratic part of the energy functional is 
the nonvanishing theorem proved in \cite[Sect.~3]{evequoz-weth-dual}.

For simplicity, and since the next result is independent of $\eps$,
we drop the subscript $\eps$.
\begin{lemma}\label{lem:splitting}
Suppose $Q\equiv Q(0)>0$ on $\R^N$.
Consider for some $d>0$ a (PS)$_d$-sequence $(v_n)_n\subset L^{p'}(\R^N)$ for $J$. 
Then there is an integer $m\geq 1$, 
critical points $w^{(1)}, \ldots, w^{(m)}$ of $J$ 
and sequences $(x_n^{(1)})_n, \ldots, (x_n^{(m)})_n\subset \R^N$ 
such that (up to a subsequence)
\begin{equation}
\left\{\begin{aligned}
 &\Bigl\| v_n-\sum_{j=1}^m w^{(j)}(\cdot+x_n^{(j)})\Bigr\|_{p'}\to 0, \quad\text{ as }n\to\infty\\
 &|x_n^{(i)}-x_n^{(j)}|\to\infty\quad\text{ as }n\to\infty, \text{ if }i\neq j,\\
 &\sum_{j=1}^m J(w^{(j)})=d.
\end{aligned}
\right.
\end{equation}
\end{lemma}
\begin{proof}
Since $(v_n)_n$ is a (PS)$_d$-sequence for $J$, it is bounded and there holds
$$
\lim_{n\to\infty}\int_{\R^N}Q^\frac1pv_n\mR\left(Q^\frac1pv_n\right)\, dx
=\frac{2p'}{2-p'}\lim_{n\to\infty}\left[J(v_n)-\frac1{p'}J'(v_n)v_n\right]=\frac{2p'd}{2-p'}>0.
$$
By the nonvanishing theorem~\cite[Theorem 3.1]{evequoz-weth-dual}, there are $R, \zeta>0$ and 
a sequence $(x_n^{(1)})_n$ such that, up to a subsequence,
$$
\int_{B_R(x_n^{(1)})}|v_n|^{p'}\, dx\geq \zeta >0\quad\text{for all }n.
$$
Replacing $(v_n)_n$ by the corresponding subsequence 
and setting $v_n^{(1)}=v_n(\cdot-x_n^{(1)})$, we find that $(v_n^{(1)})_n$ is also a (PS)$_d$-sequence
for $J$, since this functional is invariant under translations. By Lemma~\ref{lem:PS_sequences}, 
going to a further subsequence, we may assume
$v_n^{(1)}\weakto w^{(1)}$ weakly, $1_{B_R}v_n^{(1)}\to 1_{B_R}w_n^{(1)}$ strongly 
in $L^{p'}(\R^N)$, and $J(w^{(1)})\leq \lim\limits_{n\to\infty}J(v_n^{(1)})=d$. 
These last properties and the definition of $v_n^{(1)}$
imply that $w^{(1)}$ is a nontrivial critical point of $J$.

If $J(w^{(1)})=d$, we obtain
\begin{align*}
\left(\frac1{p'}-\frac12\right)\|w^{(1)}\|_{p'}^{p'}&=J(w^{(1)})-\frac12J'(w^{(1)})w^{(1)}\\
&=d=\lim_{n\to\infty}\left[J(v_n)-\frac12J'(v_n)v_n\right]
=\left(\frac1{p'}-\frac12\right)\lim_{n\to\infty}\|v_n\|_{p'}^{p'},
\end{align*}
i.e., $v_n^{(1)}\to w^{(1)}$ strongly in $L^{p'}(\R^N)$, and the lemma is proved.

Otherwise, $J(w^{(1)})<d$ and we set $v_n^{(2)}=v_n^{(1)}-w^{(1)}$.
The weak convergence $v_n^{(1)}\weakto w^{(1)}$ then implies
$$
\int_{\R^N}Q^\frac1pv_n^{(2)}\mR\left(Q^\frac1p v_n^{(2)}\right)\, dx 
= \int_{\R^N}Q^\frac1pv_n^{(1)}\mR\left(Q^\frac1p v_n^{(1)}\right)\, dx 
- \int_{\R^N}Q^\frac1pw^{(1)}\mR\left(Q^\frac1p w^{(1)}\right)\, dx +o(1),
$$
as $n\to\infty$. Moreover, by the Br\'ezis-Lieb Lemma \cite{brezis-lieb83},
$$
\int_{\R^N}|v_n^{(2)}|^{p'}\, dx =\int_{\R^N}|v_n^{(1)}|^{p'}\, dx 
-\int_{\R^N}|w^{(1)}|^{p'}\, dx+o(1),\quad \text{ as }n\to\infty.
$$
These properties and the translation invariance of $J$ together give
$$
J(v_n^{(2)})=J(v_n^{(1)})-J(w^{(1)})+o(1)=d-J(w^{(1)})+o(1), \quad\text{as }n\to\infty.
$$
Since by Lemma~\ref{lem:PS_sequences}, $1_{B_r}v_n^{(1)}\to 1_{B_r}w^{(1)}$ 
strongly in $L^{p'}(\R^N)$ for all $r>0$,
we find
$$
1_{B_r}|v_n^{(2)}|^{p'-2}v_n^{(2)}-1_{B_r}|v_n^{(1)}|^{p'-2}v_n^{(1)} 
+ 1_{B_r}|w^{(1)}|^{p'-2}w^{(1)}\to 0\quad\text{ in }L^p(\R^N), \quad\text{as }n\to\infty.
$$
Furthermore, since
$\bigl| |a|^{q-1}a-|b|^{q-1}b\bigr|\leq 2^{1-q}|a-b|^q$ for all $a, b\in\R$
and $1<q<2$, it follows that
$$
\int_{\R^N\backslash B_r}\left| |v_n^{(2)}|^{p'-2}v_n^{(2)} - |v_n^{(1)}|^{p'-2}v_n^{(1)} \right|^p\, dx
\leq 2^{(2-p')p}\int_{\R^N\backslash B_r}|w^{(1)}|^{p'}\, dx\to 0,
\quad\text{ as }r\to\infty,
$$
uniformly in $n$. Combining these two facts, we arrive at the strong convergence
$$
|v_n^{(2)}|^{p'-2}v_n^{(2)}-|v_n^{(1)}|^{p'-2}v_n^{(1)} + |w^{(1)}|^{p'-2}w^{(1)}\to 0
\quad\text{ in } L^p(\R^N), \quad\text{as }n\to\infty,
$$
and therefore,
$$
J'(v_n^{(2)})=J'(v_n^{(1)})-J'(w^{(1)})+o(1)=o(1), \quad\text{as }n\to\infty.
$$
We conclude that $(v_n^{(2)})_n$ is a PS-sequence for $J$ at level $d-J(w^{(1)})>0$. 
Thus, the nonvanishing theorem
give the existence of $R_1, \zeta_1>0$ and of a sequence $(y_n)_n\subset\R^N$
such that, going to a subsequence,
$$
\int_{B_{R_1}}|v_n^{(2)}|^{p'}\, dx\geq \zeta_1>0\quad\text{ for all }n.
$$
By Lemma~\ref{lem:PS_sequences}, there is a critical point $w^{(2)}$ of $J$ such that 
(taking a further subsequence) $v_n^{(2)}(\cdot-y_n)\weakto w^{(2)}$ weakly
and $1_Bv_n^{(2)}(\cdot-y_n)\to 1_Bw^{(2)}$ strongly in $L^{p'}(\R^N)$, for all bounded 
and measurable sets $B\subset\R^N$. In particular,
$w^{(2)}\neq 0$ and since $v_n^{(2)}\weakto 0$ we see that $|y_n|\to\infty$, as $n\to\infty$.

Setting $x_n^{(2)}=x_n^{(1)}+y_n$, we obtain $|x_n^{(2)}-x_n^{(1)}|\to\infty$ as $n\to\infty$, and
$$
v_n-\left(w^{(1)}(\cdot+x_n^{(1)})+w^{(2)}(\cdot+x_n^{(2)})\right)=v_n^{(2)}(\cdot-y_n+x_n^{(2)})
-w^{(2)}(\cdot+x_n^{(2)})\weakto 0\quad\text{ in }L^{p'}(\R^N),
$$
In addition, the same arguments as before show that 
$J(w^{(2)})\leq \liminf\limits_{n\to\infty}J(v_n^{(2)})=d-J(w^{(1)})$ 
with equality if and only if $v_n^{(2)}(\cdot-y_n)\to w^{(2)}$ strongly in $L^{p'}(\R^N)$. 
If the inequality is strict, we can iterate the procedure.
Since for every nontrivial critical point $w$ of $J$ we have $J(w)\geq c=\inf\limits_{\cN}J>0$ 
the iteration has to stop after finitely many steps, and we obtain the desired result.
\end{proof}

We now turn to investigate the Palais-Smale condition for $J_\eps$ 
and first note that if $Q(x)\to 0$ as $|x|\to\infty$,
it holds at every level, i.e., every Palais-Smale sequence has 
a convergent subsequence (see \cite[Sect.~5]{evequoz-weth-dual}).
To treat the case where
\begin{equation}\label{eqn:q_inf}
Q_\infty:=\limsup\limits_{|x|\to\infty}Q(x)>0,
\end{equation}
we consider the energy functional $J_\infty$: $L^{p'}(\R^N)$ $\to$ $\R$
given by
$$
J_\infty(v)=\frac1{p'}\int_{\R^N}|v|^{p'}\, dx 
-\frac12\int_{\R^N}Q_\infty^\frac1pv\mR(Q_\infty^\frac1pv)\, dx, 
\quad v\in L^{p'}(\R^N).
$$
The corresponding Nehari manifold
$$
\cN_\infty:=\{v\in L^{p'}(\R^N)\backslash\{0\}\ :\ J'_\infty(v)v=0\},
$$
has the same structure as $\cN_\eps$ and, since $Q_\infty$ is constant, Lemma~\ref{lem:nehari_mp} 
implies that $c_\infty:=\inf\limits_{\cN_\infty}J_\infty$ is attained and coincides with
the least energy level for nontrivial critical points of $J_\infty$.
As the last result in this section shows, the Palais-Smale condition holds for $J_\eps$ on the 
Nehari manifold $\cN_\eps$ at every energy level strictly below $c_\infty$.
The proof is inspired by the papers of Cingolani and Lazzo \cite{cingolani-lazzo97,cingolani-lazzo00}.
A new feature here is the fact that the quadratic part of the functional is nonlocal and this induces 
a nonzero interaction between functions with disjoint supports. In order to handle this, 
we first prove an estimate on this nonlocal interaction in terms of the distance between the 
supports of the two functions. It is based on a decomposition of the fundamental solution already 
introduced in \cite[Sect.~3]{evequoz-weth-dual}. Having obtained the estimate, we establish
the Palais-Smale condition for $J_\eps$ on $\cN_\eps$ below the level $c_\infty$.
\begin{lemma}\label{lem:interaction}
There exists a constant $C=C(N,p)>0$ such that 
for any $R>0$, $r\geq 1$ and $u, v\in L^{p'}(\R^N)$ 
with $\text{supp}(u)\subset B_R$
and $\text{supp}(v)\subset\R^N\backslash B_{R+r}$,
\begin{equation*}
\left|\int_{\R^N}u\mR v\, dx\right|\leq C r^{-\lambda_p}\|u\|_{p'}\|v\|_{p'}, 
\quad\text{ where }\ \lambda_p=\frac{N-1}{2}-\frac{N+1}{p}.
\end{equation*}
\end{lemma}
\begin{proof}
We prove the lemma for the nonlocal term $\int_{\R^N}v\cR u\, dx$, 
where $\cR$ denotes the resolvent operator given (for Schwartz functions) 
by the convolution with the kernel $\Phi$ in \eqref{eqn:phi} 
(see \cite[Sect.~2]{evequoz-weth-dual} for more details). 
Since $\mR$ is the real part of $\cR$ and since $u, v$ are real-valued, 
this will imply the desired result.
By density, it suffices to prove the estimate for Schwartz functions.
Let $M_{R+r}:= \R^N \backslash B_{R+r}$ and 
let $u, v\in\cS(\R^N)$ be such that $\text{supp}(u)\subset B_R$
and $\text{supp}(v)\subset M_{R+r}$. The symmetry of the operator $\cR$ and
H\"older's inequality give
\begin{equation}\label{eqn:estimR_hoelder}
\left|\int_{\R^N}u\cR v\, dx\right|=\left|\int_{\R^N}v\cR u\, dx\right|
\leq \|v\|_{p'}\|\Phi\ast u\|_{L^p(M_{R+r})}
\end{equation}
and it remains to estimate the second factor on the right-hand side.
For this we decompose $\Phi$ as follows.
Fix $\psi \in \cS(\R^N)$ such that $\wh{\psi}\in\cC^\infty_c(\R^N)$ 
is radial, $0\leq \wh{\psi}\leq 1$, $\wh{\psi}(\xi)=1$ for $| |\xi|-1|\leq\frac16$ 
and $\wh{\psi}(\xi)=0$ for $| |\xi|-1|\geq\frac14$. Writing
$\Phi= \Phi_1 + \Phi_2$
with 
$$
\Phi_1:= (2\pi)^{\frac{N}2}(\psi * \Phi), \qquad \Phi_2 = \Phi-\Phi_1,
$$
we recall the following estimates obtained in \cite{evequoz-weth-dual,e-helm-2d}:
\begin{align}
&|\Phi_1(x)| \le C_0 (1+|x|)^{\frac{1-N}{2}} \qquad
  \text{for }x \in \R^N  \label{eq:15} \\
\text{ and }\quad
&|\Phi_2(x)|\leq C_0|x|^{-N} \qquad\text{for }x\neq 0.  \label{eq:17}
\end{align}
Since the support of $u$ in contained in $B_R$, we find
\begin{align*}
\|\Phi_2 \ast u\|_{L^p(M_{R+r})} 
&\leq \left[\int_{|x| \geq R+r} \Bigl(\int_{|y| \leq R} |\Phi_2(x-y)| |u(y)|\,dy\Bigr)^p dx \right]^{\frac1p} \\
&\leq \left[\int_{\R^N} \Bigl(\int_{|x-y| \geq r} |\Phi_2(x-y)| |u(y)|\,dy\Bigr)^p dx\right]^{\frac1p} \\
&= \|(1_{M_r} |\Phi_2|) \ast |u|\ \|_p 
 \leq \|1_{M_r} \Phi_2\|_{\frac{p}2} \|u\|_{p'}.
\end{align*}
Moreover, \eqref{eq:17} gives
$$
\|1_{M_r} \Phi_2\|_{\frac{p}2} 
\leq C_0\left(\omega_N\int_r^\infty s^{N-1-\frac{Np}2}\,ds\right)^{\frac2p} 
\leq C r^{-\frac{N(p-2)}{p}}\leq C r^{-\lambda_p},
$$ 
since $r\geq 1$,
and therefore 
\begin{equation}\label{eqn:Phi2}
\|\Phi_2 \ast u\|_{L^p(M_{R+r})}\leq C r^{-\lambda_p}\|u\|_{p'}.
\end{equation}
To prove the estimate for $\Phi_1$, let us fix a radial function $\phi \in \cS(\R^N)$ such that
$\wh{\phi}\in\cC^\infty_c(\R^N)$ is radial, $0\leq \wh{\phi}\leq 1$,
$\wh{\phi}(\xi)=1$ for $| |\xi|-1|\leq \frac14$ and $\wh{\phi}(\xi)=0$
for $| |\xi|-1|\geq\frac12$. Moreover, let $\tilde u:= \phi \ast u \in\cS(\R^N)$. 
We then have $\Phi_1 \ast u = (2\pi)^{-\frac{N}2}(\Phi_1 \ast \tilde u)$, since
$\wh{\Phi_1}\wh{\phi}= (2\pi)^{\frac{N}2}\wh{\Phi_1}$ by construction. We now write 
$$
\Phi_1 \ast \tilde u = [1_{B_{\frac{r}2}} \Phi_1] \ast \tilde u 
+ [1_{M_{\frac{r}2}} \Phi_1] \ast \tilde u
$$
and let $g_r:= [1_{B_{\frac{r}2}} \Phi_1] \ast \phi$. 
Since $\text{supp}(u)\subset B_R$, we find as above
\begin{align*}
\|[1_{B_{\frac{r}2}} \Phi_1] \ast \tilde u\|_{L^p(M_{R+r})}
= \|g_r \ast  u\|_{L^p(M_{R+r})}
\leq \|(1_{M_r} |g_r|)\ast |u|\ \|_p \le 
\|1_{M_r} g_r\|_{\frac{p}{2}} \|u\|_{p'}.
\end{align*}
Using \eqref{eq:15} and the fact that $\phi \in \cS(\R^N)$, we may estimate 
\begin{align*}
\|1_{M_r} g_r\|_{\frac{p}2}^{\frac{p}2} 
&\leq C_0^{\frac{p}2} \int_{|x| \geq r} \Bigl(\int_{|y| \leq \frac{r}2}|\phi(x-y)|\,dy\Bigr)^{\frac{p}2}\,dx \\
&\leq C \int_{|x| \geq r} \Bigl(\int_{|y| \leq \frac{r}2}|x-y|^{-m}\,dy\Bigr)^{\frac{p}2} dx  
\leq C |B_{\frac{r}2}|^{\frac{p}2} \int_{|x| \geq r}\Bigl(|x|-\frac{r}2\Bigr)^{- \frac{m p}2}dx\\
&= C r^{\frac{(N-m)p}2+N} \int_{|z| \geq 1}\left(|z|-\frac{1}{2}\right)^{-\frac{mp}2}\,dz 
= C r^{\frac{(N-m)p}2+N},  
\end{align*}
where $C$ is independent of $r$ and where $m$ may be fixed so large that
$\frac{(m-N)p}2-N\geq \lambda_p$. As a consequence of
of \cite[Proposition 3.3]{evequoz-weth-dual} we have moreover
$$
\|[1_{M_{\frac{r}2}}\Phi_1] \ast \tilde u\|_{L^p(M_{R+r})}
\leq \|[1_{M_{\frac{r}2}}\Phi_1] \ast \tilde u\|_p  
\leq C r^{-\lambda_p} \|\tilde u\|_{p'}
\leq C r^{-\lambda_p} \|u\|_{p'}
$$
and we conclude that
\begin{equation}\label{eqn:Phi1}
\|\Phi_1 \ast u\|_{L^p(M_{R+r})}\leq C r^{-\lambda_p}\|u\|_{p'}.
\end{equation}
Combining \eqref{eqn:estimR_hoelder}, \eqref{eqn:Phi2} and \eqref{eqn:Phi1} yields the claim.
\end{proof}
\begin{lemma}\label{lem:PS_J_eps}
Let $\eps>0$ and assume $Q_\infty>0$ and $c_\eps<c_\infty$. 
Then $J_\eps$ satisfies the Palais-Smale condition on $\cN_\eps$ at every level
below $c_\infty$, 
i.e., every sequence $(v_n)_n\subset \cN_\eps$ such that
$J_\eps(v_n)\to d<c_\infty$ and $(J_\eps|_{\cN_\eps})'(v_n)\to 0$ as $n\to\infty$ 
has a convergent subsequence.
\end{lemma}
\begin{proof}
First note that by assumption,  $\{v\in\cN_\eps\, :\, J_\eps(v)<c_\infty\}$ is not empty. 
If $d<c_\eps$ there is nothing to prove. 
Let therefore $c_\eps\leq d<c_\infty$ and consider a (PS)$_d$-sequence $(v_n)_n$ for $J_\eps|_{\cN_\eps}$.
Since $\cN_\eps$ is a natural constraint and a $C^1$-manifold, we find that $(v_n)_n$ is a (PS)$_d$-sequence
for the unconstrained functional $J_\eps$. Using Lemma~\ref{lem:PS_sequences}, we obtain that 
(up to a subsequence) $v_n\weakto v$ and $1_{B_R} v_n\to 1_{B_R} v$ in $L^{p'}(\R^N)$ for all $R>0$, where
$v\in L^{p'}(\R^N)$ is a critical point of $J_\eps$ with $J_\eps(v)\leq d$. In order to conclude that 
$v_n\to v$ strongly in $L^{p'}(\R^N)$, it suffices to show that
\begin{equation}\label{eqn:strong_conv}
\forall\; \zeta>0, \quad \exists\; R>0\quad\text{such that}\quad
\int_{|x|>R}|v_n|^{p'}\, dx<\zeta, \quad\forall\; n.
\end{equation}
As a first step, we claim that this holds true in annular regions, in the following sense:
\begin{equation}\label{eqn:annular}
\forall\; \eta>0, \text{ and } \forall\; R>0, \quad\exists\ r>R \quad\text{such that}\quad
\liminf_{n\to\infty}\int_{r<|x|<2r}|v_n|^{p'}\, dx<\eta.
\end{equation}
Suppose not, then we find $\eta_0, R_0>0$ with the property that for every $m>R_0$ there is $n_0=n_0(m)$
such that $\int_{m<|x|<2m}|v_n|^{p'}\, dx\geq \eta_0$ for all $n\geq n_0$. Without loss
of generality, we assume that $n_0(m+1)\geq n_0(m)$ for all $m$. 
Hence, for every $\ell\in\N$ there is $N_0=N_0(\ell)$ such that
$$
\int_{\R^N}|v_n|^{p'}\, dx\geq \sum_{k=0}^{\ell-1}\int_{2^k([R_0]+1)<|x|<2^{k+1}([R_0]+1)}|v_n|^{p'}\, dx
\geq \ell\eta_0,
\quad \forall\; n\geq N_0.
$$
Letting $\ell\to\infty$, we obtain a contradiction to the fact that $(v_n)_n$ is bounded and
this gives \eqref{eqn:annular}.

\bigskip

We now prove \eqref{eqn:strong_conv} by contradiction. Assuming that it does not hold, we find $\zeta_0>0$ 
and a subsequence $(v_{n_k})_k$ such that
\begin{equation}\label{eqn:not_strong}
\int_{|x|>k}|v_{n_k}|^{p'}\, dx\geq\zeta_0, \quad\forall\; k.
\end{equation}
Let $0<\eta<\min\{1,(\frac{\zeta_0}{3C_1})^{p'}\}$ be fixed, where 
$C_1=2C(N,p)\|Q\|_\infty^\frac2p \max\{1,\sup\limits_{k\in\N}\|v_{n_k}\|_{p'}^2\}$, 
the constant $C(N,p)$ being chosen such that Lemma~\ref{lem:interaction} holds and 
$\|\mR v\|_p\leq C(N,p)\|v\|_{p'}$ for all $u\in L^{p'}(\R^N)$. 
By definition of $Q_\infty$ and since $\eps>0$ is fixed,
there is $R(\eta)>0$ such that 
$$
Q_\eps(x)\leq Q_\infty+\eta\quad\text{ for all }|x|\geq R(\eta).
$$
Also, from \eqref{eqn:annular}, we can find $r>\max\{R(\eta), \eta^{-\frac1{\lambda_p}}\}$ 
and a subsequence, still denoted by $(v_{n_k})_k$, such that
$$
\int_{r<|x|<2r}|v_{n_k}|^{p'}\, dx<\eta\quad\text{ for all }k.
$$
Setting $w_{n_k}:=1_{\{|x|\geq 2r\}}v_{n_k}$
we can write for all $k$,
\begin{align*}
&\Bigl|J_\eps'(v_{n_k})w_{n_k}-J_\eps'(w_{n_k})w_{n_k}\Bigr|
=\Bigl|\int_{|x|<r}Q_\eps^\frac1pv_{n_k}\mR(Q_\eps^\frac1pw_{n_k})\, dx
+\int_{r<|x|<2r} Q_\eps^\frac1pv_{n_k}\mR(Q_\eps^\frac1pw_{n_k})\, dx\Bigr|\\
&\leq  C(N,p) r^{-\lambda_p}\|Q\|_\infty^\frac2p\|v_{n_k}\|_{p'}^2
+ C(N,p) \|Q\|_\infty^\frac2p \|v_{n_k}\|_{p'}\left(\int_{r<|x|<2r}|v_{n_k}|^{p'}\, dx\right)^{\frac1{p'}}\\
&\leq C_1 \eta^{\frac1{p'}},
\end{align*}
using Lemma~\ref{lem:interaction}.
Also, by \eqref{eqn:not_strong} and the definition of $w_{n_k}$, there holds
$$
\int_{\R^N}|w_{n_k}|^{p'}\, dx \geq \zeta_0\quad\text{for all }k\geq 2r.
$$
Recalling our choice of $\eta$, we know that $C_1\eta^{\frac1{p'}}<\frac{\zeta_0}{3}$ and
we find some $k_0=k_0(r, \eta, \zeta)\geq 2r$ such that
\begin{equation}\label{eqn:bilinear}
\begin{aligned}
\int_{\R^N}Q_\eps^\frac1pw_{n_k}\mR(Q_\eps^\frac1pw_{n_k})\, dx
&=\int_{\R^N}|w_{n_k}|^{p'}\, dx -J'_\eps(v_{n_k})w_{n_k}+[J'_\eps(v_{n_k})w_{n_k}-J'_\eps(w_{n_k})w_{n_k}]\\
&\geq \int_{\R^N}|w_{n_k}|^{p'}\, dx -|J'_\eps(v_{n_k})w_{n_k}|-C_1\eta^{\frac1{p'}}\\
&\geq \frac{\zeta_0}{2},\quad \text{for all }k\geq k_0,
\end{aligned}
\end{equation}
since $J'_\eps(v_{n_k})w_{n_k}\to 0$ as $k\to\infty$. 
We note also that, since $v_{n_k}\in\cN_\eps$, there holds
\begin{equation}\label{eqn:p_mass}
\int_{\R^N}|w_{n_k}|^{p'}\, dx\leq \int_{\R^N}|v_{n_k}|^{p'}\, dx 
= \left(\frac1{p'}-\frac12\right)^{-1}J_\eps(v_{n_k}).
\end{equation}
For $k\geq k_0$, let now $\tilde{w}_k:=\left(\frac{Q_\eps}{Q_\infty}\right)^{\frac1p}w_{n_k}$ 
and notice that $|\tilde{w}_k|\leq \left(1+\frac{\eta}{Q_\infty}\right)^{\frac1p}|w_{n_k}|$.
In view of \eqref{eqn:bilinear}, there is $t_k^\infty>0$ for which $t_k^\infty\tilde{w}_k\in\cN_\infty$ 
and there holds
\begin{align*}
(t_k^\infty)^{2-p'}&\leq \frac{\left(1+\frac{\eta}{Q_\infty}\right)^{p'-1}\int_{\R^N}|w_{n_k}|^{p'}\, dx}{
	\int_{\R^N}Q_\eps^\frac1p w_{n_k}\mR(Q_\eps^\frac1p w_{n_k})\, dx}\\
&\leq \left(1+\frac{\eta}{Q_\infty}\right)^{p'-1}
\left(1+\frac{|J'_\eps(v_{n_k})w_{n_k}|+C_1\eta^{\frac1{p'}}}{
	\int_{\R^N}Q_\eps^\frac1p w_{n_k}\mR(Q_\eps^\frac1p w_{n_k})\, dx}\right)\\
&\leq \left(1+\frac{\eta}{Q_\infty}\right)^{p'-1}
\left(1+\frac{2|J'_\eps(v_{n_k})w_{n_k}|+2C_1\eta^{\frac1{p'}}}{\zeta_0}\right).
\end{align*}
Consequently, the above estimate and \eqref{eqn:p_mass} together give for all $k\geq k_0$,
\begin{align*}
c_\infty&\leq J_\infty(t_k^\infty\tilde{w}_k)\\
&\leq \left(\frac1{p'}-\frac12\right)(t_k^\infty)^{p'}\left(1+\frac{\eta}{Q_\infty}\right)^{p'-1}
	\int_{\R^N}|w_{n_k}|^{p'}\, dx\\
&\leq \left(1+\frac{\eta}{Q_\infty}\right)^{\frac{2(p'-1)}{2-p'}}
\left(1+\frac{2|J'_\eps(v_{n_k})w_{n_k}|+2C_1\eta^{\frac1{p'}}}{\zeta_0}\right)^{p'} J_\eps(v_{n_k}).
\end{align*}
Letting $k\to\infty$, we find
$$
c_\infty\leq \left(1+\frac{\eta}{Q_\infty}\right)^{\frac{2(p'-1)}{2-p'}}
\left(1+\frac{2C_1\eta^{\frac1{p'}}}{\zeta_0}\right)^{p'}d,
$$
and letting $\eta\to 0$ we obtain
$$
c_\infty\leq d,
$$
which contradicts the assumption $d<c_\infty$ and proves \eqref{eqn:strong_conv}.
From this we deduce the strong convergence $v_n\to v$ in $L^{p'}(\R^N)$ and the assertion follows.
\end{proof}
\begin{remark}
Under the stronger assumption $Q_\infty=\lim\limits_{|x|\to\infty}Q(x)$, the proof of the
preceding result simplifies. Indeed, having extracted a weakly converging subsequence and a critical point 
$v$ of $J_\eps$, the sequence $w_n=v_n-v$ can be shown to be a Palais-Smale sequence for $J_\infty$
at a level lying strictly below $c_\infty$. The representation lemma (Lemma~\ref{lem:splitting}) can then 
be used to conclude that $w_n\to 0$ strongly in $L^{p'}(\R^N)$.
\end{remark}

\section{Existence and concentration of dual ground states}
In this and the next section, we work under the following assumptions on $Q$.
\begin{itemize}
\item[(Q0)] $Q$ is continuous, bounded and $Q\geq 0$ on $\R^N$;
\item[(Q1)] $Q_\infty:=\limsup\limits_{|x|\to\infty}Q(x)<Q_0:=\sup\limits_{x\in\R^N}Q(x)$.
\end{itemize}
Consider the functional
$$
J_0(v):=\frac1{p'}\int_{\R^N}|v|^{p'}\, dx -\frac12\int_{\R^N}Q_0^\frac1pv\mR(Q_0^\frac1pv)\, dx, 
\quad v\in L^{p'}(\R^N)
$$
and the corresponding Nehari manifold
$$
\cN_0:=\{v\in L^{p'}(\R^N)\backslash\{0\}\ :\ J'_0(v)v=0\}
$$
associated to the limit problem
\begin{equation}\label{eqn:sp_lim}
-\Delta u -u = Q_0|u|^{p-2}u, \quad x\in\R^N.
\end{equation}
Lemma~\ref{lem:nehari_mp} implies that the level $c_0:=\inf\limits_{\cN_0}J_0$ 
is attained and coincides with the least-energy level, i.e.,
$$
c_0=\inf\{J_0(v)\, :\, v\in L^{p'}(\R^N), \ v\neq 0\text{ and }J_0'(v)=0\}.
$$
Our first goal will be to show, comparing the energy level $c_\eps$ with $c_0$, that for small $\eps>0$, 
$c_\eps$ is attained.
For this, let us denote the set of maximum points of $Q$ by
$$
M:=\{x\in\R^N\, :\, Q(x)=Q_0\}.
$$
Notice that $M\neq\varnothing$, since (Q0) and (Q1) are assumed.
We start by studying the projection on the Nehari manifold of truncations 
of translated and rescaled ground states of $J_0$.
Take a cut-off function $\eta\in C^\infty_c(\R^N)$, $0\leq\eta\leq 1$, 
such that $\eta\equiv 1$ in $B_1(0)$ and $\eta\equiv 0$ in $\R^N\backslash B_2(0)$.
For $y\in M$, $\eps>0$ we let
\begin{equation}\label{eqn:phi_eps_y}
\varphi_{\eps,y}(x):=\eta(\eps x-y)\ w( x-\eps^{-1}y),
\end{equation}
where $w\in L^{p'}(\R^N)$ is some fixed least-energy critical point of $J_0$.
\begin{lemma}\label{lem:phi}
There is $\eps^\ast>0$ such that 
for all $0<\eps\leq\eps^\ast$, $y\in M$, a unique
$t_{\eps,y}>0$ satisfying $t_{\eps,y}\varphi_{\eps,y}\in\cN_\eps$ exists. Moreover,
$$
\lim_{\eps\to 0^+}J_\eps(t_{\eps,y}\varphi_{\eps,y})=c_0, \text{ uniformly for }y\in M.
$$
\end{lemma}
\begin{proof}
We start by remarking that $Q(y+\eps\cdot)\eta(\eps\cdot)w\to w$ in $L^{p'}(\R^N)$, 
uniformly with respect to $y\in M$, since $M$ is compact
and $Q$ is continuous by assumption. Consequently,
\begin{align*}
 \int_{\R^N}Q_\eps^\frac1p\varphi_{\eps,y}\mR(Q_\eps^\frac1p\varphi_{\eps,y})\, dx
 &=\int_{\R^N}Q^\frac1p(y+ \eps z)\eta(\eps z)w(z)\mR(Q^\frac1p(y+\eps\cdot)\eta(\eps\cdot)w)(z)\, dz\\
 &\longrightarrow \int_{\R^N}Q_0^\frac1pw\mR(Q_0^\frac1pw)\, dz=\left(\frac1{p'}-\frac12\right)^{-1}c_0>0,
  \quad\text{as }\eps\to 0^+,
\end{align*}
uniformly for $y\in M$. Therefore, $\varphi_{\eps,y}\in U_\eps^+$ for all $y\in M$ and
$\eps>0$ small enough, which shows the first assertion with $t_{\eps,y}$ given by \eqref{eqn:t_v}. 
In addition, for all $y\in M$,
$$
\int_{\R^N}|\varphi_{\eps,y}|^{p'}\, dx=\int_{\R^N}|\eta(\eps z)w(z)|^{p'}\, dz \to \int_{\R^N}|w|^{p'}\, dz
=\left(\frac1{p'}-\frac12\right)^{-1}c_0,\quad\text{as }\eps\to 0^+.
$$
As a consequence, $t_{\eps, y}\to 1$ as $\eps\to 0^+$, uniformly for $y\in M$, and
we obtain $J_\eps(t_{\eps, y}\varphi_{\eps,y})\to c_0$ as $\eps\to 0^+$, uniformly for $y\in M$. 
The second assertion follows.
\end{proof}
\begin{lemma}\label{lem:c>c_0}
For all $\eps>0$ there holds $c_\eps\geq c_0$.
Moreover,
$\lim\limits_{\eps\to 0^+}c_\eps=c_0$.
\end{lemma}
\begin{proof}
Consider $v_\eps\in\cN_\eps$ and set $v_0:=\left(\frac{Q_\eps}{Q_0}\right)^\frac1p v_\eps$.
Notice that (Q1) implies $|v_0|\leq |v_\eps|$ a.e. on $\R^N$. 
Since $v_\eps\in U^+_\eps$, we find
$$
\int_{\R^N}Q_0^\frac1p v_0\mR(Q_0^\frac1p v_0)\, dx =
\int_{\R^N}Q_\eps^\frac1pv_\eps\mR(Q_\eps^\frac1p v_\eps)\, dx>0,
$$
i.e. $v_0\in U^+_0$. Hence, with
$$
t_\eps^{2-p'}=\frac{\int_{\R^N}|v_0|^{p'}\, dx}{\int_{\R^N}Q_0^\frac1p v_0\mR(Q_0^\frac1p v_0)\, dx}
\leq \frac{\int_{\R^N}|v_\eps|^{p'}\, dx}{\int_{\R^N}Q_\eps^\frac1p v_\eps\mR(Q_\eps^\frac1p v_\eps)\, dx}=1,
$$
it follows that $t_\eps v_0\in \cN_0$, and we obtain
$$
c_0\leq J_0(t_\eps v_0)=\left(\frac1{p'}-\frac12\right) t_\eps^{p'}\int_{\R^N}|v_0|^{p'}\, dx
\leq \left(\frac1{p'}-\frac12\right) \int_{\R^N}|v_\eps|^{p'}\, dx=J_\eps(v_\eps).
$$
Since $v_\eps\in\cN_\eps$ was arbitrarily chosen, we conclude that 
$c_\eps=\inf\limits_{\cN_\eps}J_\eps\geq c_0$.
On the other hand, Lemma~\ref{lem:phi} gives for $y\in M$, 
$c_\eps\leq J_\eps(t_{\eps,y}\varphi_{\eps,y})\to c_0$ as $\eps\to 0^+$.
Hence, $\lim\limits_{\eps\to 0^+}c_\eps= c_0$ and the lemma is proved.
\end{proof}
\begin{proposition}\label{prop:c_eps_attained}
There is $\eps_0>0$ such that for all $\eps<\eps_0$ the least-energy level $c_\eps$ is attained.
\end{proposition}
\begin{proof}
By Lemma~\ref{lem:c>c_0} and Condition (Q1) there is $\eps_0>0$
such that $c_\eps<c_\infty$ for all $0<\eps<\eps_0$. For such $\eps$,
using the fact that $\cN_\eps$ is a $C^1$-submanifold of $L^{p'}(\R^N)$, 
we obtain from Ekeland's variational principle \cite[Theorem 3.1]{ekeland74} 
the existence of a Palais-Smale sequence for $J_\eps$ on $\cN_\eps$, at level $c_\eps$, 
and Lemma~\ref{lem:PS_J_eps} concludes the proof.
\end{proof}
Setting $k_0=\eps_0^{-1}$, the assertion (i) in Theorem \ref{thm1} from the Introduction 
is a direct consequence of the above result. Our next step is to examine the
behavior of critical points of $J_\eps$ in the limit $\eps\to 0$.
\begin{proposition}\label{prop:limit_gs}
Let $(\eps_n)_n\subset(0,\infty)$ satisfy $\eps_n\to 0$ as $n\to\infty$. Consider for each $n$
some $v_n\in \cN_{\eps_n}$ and assume that $J_{\eps_n}(v_n)\to c_0$ as $n\to\infty$. 
Then, there is $x_0\in M$, a critical point $w_0$ of $J_0$ at level $c_0$ and a sequence $(y_n)_n\subset\R^N$ 
such that (up to a subsequence) 
$$
\eps_ny_n\to x_0\quad\text{ and }\quad
\|v_n(\cdot+y_n)- w_0\|_{p'}\to 0\ \text{ as }\ n\to\infty.
$$
\end{proposition}
\begin{proof}
For each $n\in\N$, set $v_{0,n}:=\left(\frac{Q_{\eps_n}}{Q_0}\right)^\frac1pv_n$. It follows
that $|v_{0,n}|\leq |v_n|$ a.e. on $\R^N$ and that
$$
\int_{\R^N}Q_0^\frac1pv_{0,n}\mR(Q_0^\frac1pv_{0,n})\, dx 
= \int_{\R^N}Q_{\eps_n}^\frac1pv_n\mR(Q_{\eps_n}^\frac1pv_n)\, dx>0.
$$
Therefore, setting
$$
t_{0,n}^{2-p'}=\frac{\int_{\R^N}|v_{0,n}|^{p'}\, dx}{\int_{\R^N}Q_0^\frac1pv_{0,n}\mR(Q_0^\frac1pv_{0,n})\, dx}
$$
we find that $t_{n,0}v_{0,n}\in \cN_0$ and $0<t_{0,n}\leq 1$. As a consequence, we can write
\begin{align*}
c_0&\leq J_0(t_{0,n}v_{0,n})
=\left(\frac1{p'}-\frac12\right)t_{0,n}^2\int_{\R^N}Q_0^\frac1pv_{0,n}\mR(Q_0^\frac1pv_{0,n})\, dx\\
&=\left(\frac1{p'}-\frac12\right)t_{0,n}^2\int_{\R^N}Q_{\eps_n}^\frac1pv_n\mR(Q_{\eps_n}^\frac1pv_n)\, dx\\
&=t_{0,n}^2J_{\eps_n}(v_n)\leq J_{\eps_n}(v_n)\to c_0, \quad\text{ as }n\to\infty.
\end{align*}
In particular, we find 
$$
\lim_{n\to\infty}t_{0,n}=1,
$$
and $(t_{0,n}v_{0,n})_n\subset\cN_0$ is thus a minimizing sequence for $J_0$ on $\cN_0$.
Using Ekeland's variational principle \cite{ekeland74} and the fact that $\cN_0$ is a natural constraint,
we obtain the existence of a (PS)$_{c_0}$-sequence $(w_n)_n\subset L^{p'}(\R^N)$ for $J_0$ with the
property that $\|v_{0,n}-w_n\|_{p'}\to 0$, as $n\to\infty$.

By Lemma \ref{lem:splitting}, there exists a critical point $w_0$ for $J_0$
at level $c_0$ and a sequence $(y_n)_n\subset\R^N$ such that (up to a subsequence)
$\|w_n(\cdot+y_n)-w_0\|_{p'}\to 0$, as $n\to\infty$.
Therefore, 
$$
v_{0,n}(\cdot+y_n)\to w_0\quad\text{ strongly in }L^{p'}(\R^N),\text{ as }n\to\infty.
$$

We now claim that $(\eps_ny_n)_n$ is bounded. 
Suppose by contradiction that some subsequence (which we still call $(\eps_ny_n)_n$)
has the property $\lim\limits_{n\to\infty}|\eps_ny_n|=\infty$. We distinguish two cases.

(1) If $Q_\infty=0$, then
$Q(\eps_n \cdot+\eps_n y_n)\to 0$, as $n\to\infty$, holds uniformly on bounded sets of $\R^N$. 
From the definition of $v_{0,n}$, we infer that $v_{0,n}(\cdot+y_n)\weakto 0$ and therefore $w_0=0$,
in contradiction to $J_0(w_0)=c_0>0$. Hence, $(\eps_ny_n)_n$ is bounded in this case.

(2) If $Q_\infty>0$ instead, Fatou's lemma and the strong convergence $v_{0,n}(\cdot+y_n)\to w_0$ 
together imply
\begin{align*}
c_0&=\lim_{n\to\infty}J_{\eps_n}(v_n)
=\lim_{n\to\infty}\left(\frac1{p'}-\frac12\right)\int_{\R^N} |v_n|^{p'}\, dx\\
&=\lim_{n\to\infty}\left(\frac1{p'}-\frac12\right)\int_{\R^N} |v_n(x+y_n)|^{p'}\, dx\\
&=\liminf_{n\to\infty}\left(\frac1{p'}-\frac12\right)
\int_{\R^N}\left(\frac{Q_0}{Q(\eps_nx+\eps_ny_n)}\right)^{p'-1}|v_{0,n}(x+y_n)|^{p'}\, dx\\
&\geq\left(\frac1{p'}-\frac12\right)\int_{\R^N}\left(\frac{Q_0}{Q_\infty}\right)^{p'-1}|w_0|^{p'}\, dx\\
&=\left(\frac{Q_0}{Q_\infty}\right)^{p'-1}c_0,
\end{align*}
and this contradicts (Q1). Therefore, $(\eps_ny_n)_n$ is a bounded sequence and we may 
assume (going to a subsequence) that $\eps_n y_n\to x_0\in\R^N$. 
Since $Q(\eps_n x+\eps_ny_n)\to Q(x_0)$, as $n\to\infty$,
uniformly on bounded sets, the argument of Case (1) above gives $Q(x_0)>0$ and, using the Dominated
Convergence Theorem, we see that $Q(x_0)=Q_0$, since the following holds.
\begin{align*}
c_0&=\lim_{n\to\infty}J_{\eps_n}(v_n)
=\lim_{n\to\infty}\left(\frac1{p'}-\frac12\right)\int_{\R^N} |v_n|^{p'}\, dx\\
&=\lim_{n\to\infty}\left(\frac1{p'}-\frac12\right)
\int_{\R^N}\left(\frac{Q_0}{Q(\eps_n x+\eps_ny_n)}\right)^{p'-1}|v_{0,n}(x+y_n)|^{p'}\, dx\\
&=\left(\frac1{p'}-\frac12\right)\int_{\R^N}\left(\frac{Q_0}{Q(x_0)}\right)^{p'-1}|w_0|^{p'}\, dx\\
&=\left(\frac{Q_0}{Q(x_0)}\right)^{p'-1}c_0.
\end{align*}
Going back to the original sequence we obtain
\begin{align*}
v_n(\cdot+y_n)=\left(\frac{Q_0}{Q(\eps_n(\cdot+y_n))}\right)^\frac1p v_{0,n}(\cdot+y_n)
\to \left(\frac{Q_0}{Q(x_0)}\right)^\frac1p w_0=w_0,\quad\text{as }n\to\infty,
\end{align*}
strongly in $L^{p'}(\R^N)$, using again the Dominated Convergence Theorem.
The proof is complete.
\end{proof}
In the next result, we prove the assertion (ii) in Theorem \ref{thm1} from the Introduction.
For the reader's convenience, let us recall its formulation.
\begin{theorem}\label{thm:conc_gs}
Let $k_0:=\eps_0^{-1}>0$, where $\eps_0>0$ is given by Proposition~\ref{prop:c_eps_attained}.
For every sequence $(k_n)_n\subset(k_0,\infty)$ satisfying $k_n\to \infty$ as $n\to\infty$, 
and every sequence $(u_n)_n$ such that $u_n$ is a dual ground state of
$$
-\Delta u - k_n u=Q(x)|u|^{p-2}u \quad\text{in }\ \R^N,
$$
there is $x_0\in M$, a dual ground state $u_0$ of \eqref{eqn:sp_lim}
and a sequence $(x_n)_n\subset\R^N$ such that (up to a subsequence)
$\lim\limits_{n\to\infty}x_n= x_0$ and  
$$
k_n^{-\frac{2}{p-2}}u_n\left(\frac{\cdot}{k_n}+x_n\right)\to u_0\quad\text{ in }\ L^p(\R^N),
\ \text{ as }\ n\to\infty.
$$ 
\end{theorem}
\begin{proof}
For each $n$, the dual ground state $u_n$ can be represented as
$$
u_n(x)=k_n^{\frac2{p-2}}\mR(Q_{\eps_n}^\frac1pv_n)(k_nx),\quad x\in\R^N,
$$
where $\eps_n=k_n^{-1}$ and $v_n\in L^{p'}(\R^N)$ is a least-energy critical point of $J_{\eps_n}$, i.e., 
$J_{\eps_n}'(v_n)=0$ and $J_{\eps_n}(v_n)=c_{\eps_n}$.
By Lemma~\ref{lem:c>c_0} and Proposition~\ref{prop:limit_gs}, there is $x_0\in M$ 
and a sequence $(y_n)_n\subset\R^N$ such that, as $n\to\infty$, $x_n:=\eps_ny_n\to x_0$ and, 
going to a subsequence, $v_n(\cdot+y_n)\to w_0$ in $L^{p'}(\R^N)$ for some least-energy critical point $w_0$ of $J_0$.
Since for $x\in\R^N$,
$$
k_n^{-\frac2{p-2}}u_n\left(\frac{x}{k_n}+x_n\right)=\mR(Q_{\eps_n}^\frac1pv_n)(x+y_n)
=\mR\Bigl(Q_{\eps_n}^\frac1p(\cdot+y_n)v_n(\cdot+y_n)\Bigr)(x),
$$
we obtain, using the continuity of $\mR$ and the pointwise convergence 
$Q_{\eps_n}(x+y_n)\to Q(x_0)=Q_0$ as $n\to\infty$ for all $x\in\R^N$,
the strong convergence
$$
k_n^{-\frac2{p-2}}u_n\left(\frac{x}{k_n}+x_n\right) \to \mR\left(Q_0^\frac1pw_0\right)\quad\text{ in }L^p(\R^N).
$$
Setting $u_0=\mR(Q_0^\frac1pw_0)$, the properties $J_0(w_0)=c_0$ and $J_0'(w_0)=0$ imply 
that $u_0$ is a dual ground state solution of \eqref{eqn:sp_lim}
and this concludes the proof.
\end{proof}

\begin{remark}\label{remq:conc}
\begin{itemize} 
\item[(i)] The conclusion of the preceding theorem holds more generally for every sequence of dual bound states. 
Indeed, in view of Proposition~\ref{prop:limit_gs} it is enough to have $u_n(x)=\mR(Q_{\eps_n}^\frac1pv_n)(k_nx)$, 
where $v_n$ is a critical point of $J_{\eps_n}$, and to require $J_{\eps_n}(v_n)\to c_0$ as $n\to\infty$.
\item[(ii)] Elliptic estimates imply that the convergence towards $u_0$ holds in $W^{2,q}(\R^N)$ for all 
$\frac{2N}{N-1}<q<\infty$. In particular, the convergence holds in $L^\infty(\R^N)$ and 
since $u_0\in W^{2,p}(\R^N)$ we find that for every $\delta>0$ there is $R_\delta>0$ such that for large $n$, 
$$
k_n^{-\frac2{p-2}}|u_n(x)|< \delta\quad\text{ for all }|x-x_n|\geq\frac{R_\delta}{k_n},
$$
whereas $k_n^{-\frac2{p-2}}\|u_n\|_\infty\to \|u_0\|_\infty>0$ as $n\to\infty$. 
In addition, if $\tilde{x}_n$ denotes any global maximum point of $|u_n|$, 
then $\tilde{x}_n\to x_0$ as $n\to\infty$.
\end{itemize}
\end{remark}

\section{Multiplicity of dual bound states}
As before, we work under the assumptions (Q0) and (Q1) and let $M$ denote the set of
maximum points of $Q$. In addition, for $\delta>0$ we consider the closed neighborhood
$M_\delta:=\{x\in\R^N\, :\, \text{dist}(x,M)\leq\delta\}$ of $M$. 

The purpose of this section is to prove the multiplicity result stated in the Introduction, 
relating the number of solutions of \eqref{eqn:sp} and the topology of $M$. 
We recall it for the reader's convenience. 
\begin{theorem}\label{thm:multiple}
Suppose (Q0) and (Q1) holds. For every $\delta>0$, there exists $k(\delta)>0$ 
such that the problem \eqref{eqn:sp} has at least $\operatorname*{cat}_{M_\delta}(M)$
distinct dual bound states for all $k>k(\delta)$.
\end{theorem}
To prove this result, we shall construct two maps whose composition is homotopic to the inclusion
$M\embed M_\delta$. We start by introducing some notation.

For fixed $\delta>0$, we consider the family of rescaled barycenter type maps 
$\beta_\eps$: $L^{p'}(\R^N)\backslash\{0\}$ $\to$ $\R^N$, $\eps>0$, given as follows. 
Let $\rho>0$ be such that $M_\delta\subset B_\rho(0)$ and define $\Xi$: $\R^N$ $\to$ $\R^N$ by
$$
\Xi(x)=\left\{\begin{array}{ll} x & \text{ if }|x|<\rho \\ 
\frac{\rho x}{|x|} & \text{ if }|x|\geq \rho.\end{array}\right.
$$
For $v\in L^{p'}(\R^N)\backslash\{0\}$, we set
$$
\beta_\eps(v):=\frac1{\|v\|_{p'}^{p'}}\int_{\R^N}\Xi(\eps x)|v(x)|^{p'}\, dx.
$$
Moreover, as in the previous section, we consider for $\eps>0$ and $y\in M_\delta$ 
the function $\varphi_{\eps,y}\in L^{p'}(\R^N)$ defined by \eqref{eqn:phi_eps_y}, where
$\eta\in C^\infty_c(\R^N)$ is a cutoff function satisfying 
$0\leq\eta\leq 1$ in $\R^N$, $\eta\equiv 1$ in $B_1(0)$ and 
$\eta\equiv 0$ in $\R^N\backslash B_2(0)$, and where
$w\in L^{p'}(\R^N)$ is any fixed least-energy critical point of $J_0$.

We note that, due to the compactness of $M_\delta$, 
the following holds uniformly in $y\in M_\delta$.
\begin{equation}\label{eqn:beta_phi}
\begin{aligned}
\lim_{\eps\to 0^+}\beta_\eps(\varphi_{\eps,y})
&=\lim_{\eps\to 0^+}\frac{\int_{\R^N}\Xi(y+\eps z)\eta(\eps z)|w(z)|^{p'}\, dz}{\int_{\R^N}\eta(\eps z)|w(z)|^{p'}\, dz}
=\Xi(y)=y.
\end{aligned}
\end{equation}
Before proving the main result in this section,
we need the following preparatory lemma.
\begin{lemma}\label{lem:beta_sigma}
Let $\delta>0$ and let $\nu$: $(0,\infty)$ $\to$ $(0,\infty)$ satisfy
$\lim\limits_{\eps\to 0^+}\nu(\eps)=0$ and $\nu(\eps)>c_\eps-c_0$ for all $\eps>0$. 
Considering the sublevel set
$$
\Sigma_\eps:=\{v\in \cN_\eps\, :\, J_\eps(v)\leq c_0+\nu(\eps)\},
$$
there holds
$$
\lim_{\eps\to 0^+} \sup\limits_{v\in\Sigma_\eps}\inf\limits_{y\in M_{\frac{\delta}{2}}}|\beta_\eps(v)-y|=0.
$$
\end{lemma}
\begin{proof}
Notice that $\Sigma_\eps\neq\varnothing$, since $c_\eps<c_0+\nu(\eps)$ by assumption.
Let $(\eps_n)_n\subset(0,\infty)$ be any sequence such that $\eps_n\to 0$ as $n\to\infty$, 
and choose for each $n$ some $v_n\in\Sigma_{\eps_n}$ such that
\begin{equation}\label{eqn:inf_Sigma_Mdelta}
\inf\limits_{y\in M_{\frac{\delta}{2}}}|\beta_{\eps_n}(v_n)-y|
\geq \sup\limits_{v\in\Sigma_{\eps_n}}\inf\limits_{y\in M_{\frac{\delta}{2}}}|\beta_{\eps_n}(v)-y|-\frac1n.
\end{equation}
By Proposition~\ref{prop:limit_gs}, there is $x_0\in M$, a least-energy critical point $w_0$
of $J_0$ and a sequence $(y_n)_n\subset\R^N$ such that, up to a subsequence,
$\eps_ny_n\to x_0$ and $v_n(\cdot+y_n)\to w_0$ in $L^{p'}(\R^N)$, as $n\to\infty$.
Therefore, similar to \eqref{eqn:beta_phi} we obtain
$$
\beta_{\eps_n}(v_n)
=\frac{\int_{\R^N}\Xi(\eps_nx+\eps_ny_n)|v_n(x+y_n)|^{p'}\, dx}{\int_{\R^N}|v_n(x+y_n)|^{p'}\, dx}
\to \Xi(x_0)=x_0,\quad\text{as }n\to\infty.
$$
From \eqref{eqn:inf_Sigma_Mdelta} we deduce that (up to a subsequence)
$\sup\limits_{v\in\Sigma_{\eps_n}}\inf\limits_{y\in M_{\frac{\delta}{2}}}|\beta_{\eps_n}(v)-y|\to 0$
as $n\to\infty$. Since the sequence $(\eps_n)_n$ was arbitrarily chosen, the conclusion follows
by a contradiction argument.
\end{proof}
\begin{proof}[Proof of Theorem~\ref{thm:multiple}]
Let $\delta>0$. According to Lemma \ref{lem:phi}, Lemma \ref{lem:c>c_0} and the assumption (Q1),
we can find $\bar\eps>0$ and a function $\nu$: $(0,\infty)$ $\to$ $(0,\infty)$ 
such that $\nu(\eps)>c_\eps-c_0$ for all $\eps>0$, $\nu(\eps)\to 0$ as $\eps\to 0^+$
and $J_\eps(t_{\eps,y}\varphi_{\eps,y})<c_0+\nu(\eps)<c_\infty$, 
for all $y\in M$ and all $0<\eps<\bar\eps$. Moreover, let us assume without loss of generality
that, for every $0<\eps<\bar\eps$, the level $c_0+\nu(\eps)$ is not critical for $J_\eps$.
 
Consider for $0<\eps<\bar\eps$ the set $\Sigma_\eps$ given in Lemma \ref{lem:beta_sigma}.
Then $t_{\eps,y}\varphi_{\eps,y}\in\Sigma_\eps$ and there exists $\eps_1\leq \bar\eps$ 
such that for all $0<\eps<\eps_1$,
\begin{equation}\label{eqn:dist}
\sup\limits_{v\in\Sigma_\eps}\inf\limits_{y\in M_{\frac{\delta}{2}}}|\beta_\eps(v)-y|<\frac{\delta}{2}.
\end{equation}
In particular, $\beta_\eps(\Sigma_\eps)\subset M_\delta$ and by \eqref{eqn:beta_phi} 
the map $y\mapsto \beta_\eps(\varphi_{\eps,y})=\beta_\eps(t_{\eps,y}\varphi_{\eps,y})$ is homotopic
to the inclusion $M\embed M_\delta$ in $M_\delta$. 
Therefore, \cite[Lemma 2.2]{cingolani-lazzo00} gives 
$\text{cat}_{\Sigma_\eps}(\Sigma_\eps)\geq\text{cat}_{M_\delta}(M)$ for all $0<\eps<\eps_1$. 

Since $\cN_\eps$ is a complete $C^1$-manifold and since by Lemma \ref{lem:PS_J_eps}, 
$J_\eps$ satisfies the Palais-Smale condition on $\Sigma_\eps$, the Ljusternik-Schnirelman theory 
for $C^1$-manifolds from \cite{Rib-Tsa-Kra95} (see also \cite{cv_dg_mz93,szulkin88}) ensures 
the existence of at least $\text{cat}_{M_\delta}(M)$ distinct critical points of $J_\eps$ for all 
$0<\eps<\eps_1$.

The  transformation \eqref{eqn:u_transf_dual} gives for each critical point of $J_\eps$ 
a dual bound state of \eqref{eqn:sp} with $k=\eps^{-1}$ and,
since distinct critical points correspond to distinct bound states, the theorem 
follows by setting $k(\delta)=\eps_1^{-1}$.
\end{proof}

\begin{remark}
According to Remark~\ref{remq:conc}(i), the solutions given by Theorem~\ref{thm:multiple} concentrate as $k\to \infty$ 
in the sense of Theorem~\ref{thm:conc_gs}.
\end{remark}

\section*{Acknolwedgements}
\noindent This research is supported by the grant WE 2812/5-1 of the Deutsche Forschungsgemeinschaft (DFG).
The author would like to thank Tobias Weth for suggesting the question studied 
in this paper and for helpful discussions and valuable advice.

\bibliographystyle{abbrv}
\bibliography{literatur.bib}

\end{document}